\documentclass[11pt]{amsart}
\usepackage{latexsym, amssymb, stmaryrd}
\usepackage[T1]{fontenc}

\DeclareTextSymbol{\thh}{T1}{254}

\newtheorem{thm}{Theorem}[subsection]
\newtheorem{lemma}[thm]{Lemma}
\newtheorem{prop}[thm]{Proposition}
\newtheorem{cor}[thm]{Corollary}

\theoremstyle{definition}
\newtheorem{df}[thm]{Definition}
\newtheorem{rmk}[thm]{Remark}
\newtheorem{rmks}[thm]{Remarks}
\newtheorem{ex}[thm]{Example}
\newtheorem{fact}[thm]{Fact}

\newcommand{\BB}{\mathbb{B}}  
\newcommand{\BK}{\mathbb{K}}  
\newcommand{\BN}{\mathbb{N}}

\newcommand{\cu}[1]{\mathcal{#1}}

\newcommand{\bo}[1]{\boldsymbol{#1}}

\newcommand{\sa}[1]{\mathsf{#1}}
\renewcommand{\hat}{\widehat}

\makeatletter

\def\indsym#1#2{%
  \setbox0=\hbox{$\m@th#1x$}%
  \kern\wd0%
  \hbox to 0pt{\hss$\m@th#1\mid$\hbox to 0pt{$\m@th#1^{#2}$}\hss}%
  \lower.9\ht0\hbox to 0pt{\hss$\m@th#1\smile$\hss}%
  \kern\wd0}

\def\nindsym#1#2{%
  \setbox0=\hbox{$\m@th#1x$}%
  \kern\wd0%
  \hbox to 0pt{\hss$\m@th#1\not$\kern1.4\wd0\hss}
  \hbox to 0pt{\hss$\m@th#1\mid$\hbox to 0pt{$\m@th#1^{\,#2}$}\hss}%
  \lower.9\ht0\hbox to 0pt{\hss$\m@th#1\smile$\hss}%
  \kern\wd0}

\def\dotminussym#1#2{%
  \setbox0=\hbox{$\m@th#1-$}%
  \kern.5\wd0%
  \hbox to 0pt{\hss\hbox{$\m@th#1-$}\hss}%
  \raise.6\ht0\hbox to 0pt{\hss$\m@th#1.$\hss}%
  \kern.5\wd0}

\def \<{\langle}
\def \>{\rangle}
\def \((  {(\!(}
\def \)) {)\!)}

\def \cl {\operatorname{cl}}

\def \tp{\operatorname{tp}}

\def \acl{\operatorname{acl}}
\def \dcl{\operatorname{dcl}}

\def \int{\operatorname{int}}
\def \DLO{\operatorname{DLO}}
\def \lcl{\operatorname{lcl}}

\def \fo{\operatorname{fdcl}}
\def \range{\operatorname{range}}
\numberwithin{equation}{section}
\def \l{\llbracket}
\def \rr{\rrbracket}

\allowdisplaybreaks[2]

\begin{document}

\title{Definable Closure in Randomizations}

\author{Uri Andrews, Isaac Goldbring, and H. Jerome Keisler}

\address{University of Wisconsin-Madison, Department of Mathematics, Madison,  WI 53706-1388}
\email{andrews@math.wisc.edu}
\urladdr{www.math.wisc.edu/~andrews}
\email{keisler@math.wisc.edu}
\urladdr{www.math.wisc.edu/~keisler}
\address {University of Illinois at Chicago, Department of Mathematics, Statistics, and Computer Science,
Science and Engineering Offices (M/C 249), 851 S. Morgan St., Chicago, IL 60607-7045, USA}
\email{isaac@math.uic.edu}
\urladdr{www.math.uic.edu/~isaac}

\begin{abstract}
The randomization of a complete first order theory $T$ is the complete continuous theory $T^R$
with two sorts, a sort for random elements of models of $T$, and a sort for events in an underlying probability space.
We give necessary and sufficient conditions for an element to be definable over a set of parameters in a model of $T^R$.
\end{abstract}


\maketitle

\section{Introduction}

A randomization of a first order structure $\cu M$, as introduced by Keisler [Kei1] and formalized as a metric structure by Ben Yaacov and Keisler [BK], is a continuous structure $\cu N$ with two sorts, a sort for random elements of $\cu M$, and a sort for events in an underlying atomless probability space. Given a complete first order theory $T$, the theory $T^R$ of randomizations of models of $T$ forms a complete theory in continuous logic, which is called the randomization of $T$.  In a model $\cu N$ of $T^R$, for each $n$-tuple $\vec{ a}$ of random elements and each first order
formula $\varphi(\vec v)$, the set of points in the underlying probability space where $\varphi(\vec{ a})$ is true is an event
denoted by $\l\varphi(\vec{ a})\rr$.

In a first order structure $\cu M$, an element $b$ is \emph{definable over} a set $A$ of elements of $\cu M$ (called parameters) if there is a tuple $\vec a$ in $A$ and a formula $\varphi(u,\vec a)$ such that
$$\cu M\models (\forall u)(\varphi(u,\vec a)\leftrightarrow u=b).$$
In a general metric structure $\cu N$, an element $ b$ is said to be \emph{definable over} a set of parameters $ A$ if there is a sequence of
tuples $\vec{ a}_n$ in $ A$ and continuous formulas $\Phi_n(x,\vec{ a}_n)$ whose truth values converge uniformly to the distance from $x$
to $ b$.  In this paper we give necessary and sufficient conditions for definability in a model of the randomization theory $T^R$.  These conditions
can be stated in terms of sequences of first order formulas.  The results in this paper will be applied in a forthcoming paper about
independence relations in randomizations.

In Theorem \ref{t-definableB}, we show that an event $\sa E$ is definable over a set $ A$ of parameters if and only if it is
the limit of a sequence of events of the form $\l\varphi_n(\vec{ a}_n)\rr$, where each $\varphi_n$ is a
first order formula and each $\vec{ a}_n$ is a tuple from $ A$.

In Theorem \ref{t-separable}, we show that a random element $ b$ is definable over a set $ A$ of parameters if and only if $ b$ is
the limit of a sequence of  random elements $ b_n$ such that for each $n$,
$$\l(\forall u)(\varphi_n(u,\vec{ a}_n)\leftrightarrow u= b_n)\rr$$
has probability one for some first order formula $\varphi_n(u,\vec v)$ and a tuple $\vec{ a}_n$ from $ A$.
In Section 4 we give some consequences in the special case that the underlying first order theory $T$ is $\aleph_0$-categorical.

Continuous model theory in its current form is developed in the papers [BBHU] and [BU].  The papers [Go1], [Go2], [Go3] deal with definability questions
in metric structures.  Randomizations of models are treated in [AK], [Be], [BK], [EG], [GL], [Ke1], and [Ke2].

\section{Preliminaries}

We refer to [BBHU] and [BU] for background in continuous model theory, and follow the notation of [BK].
We assume familiarity with the basic notions about continuous model theory as developed
in [BBHU], including the notions of a theory, structure, pre-structure,  model of a theory,
elementary extension, isomorphism, and $\kappa$-saturated structure.
In particular, the universe of a pre-structure is a pseudo-metric space, the universe of a
structure is a complete metric space, and every pre-structure has a unique completion.
In continuous logic, formulas have truth values in the unit interval $[0,1]$ with $0$ meaning true, the connectives are
continuous functions from $[0,1]^n$ into $[0,1]$, and the quantifiers are $\sup$ and $\inf$.
A \emph{tuple} is a finite sequence, and $A^{<\BN}$ is the set of all tuples of elements of $A$.

\subsection{The theory $T^R$}

We assume throughout that $L$ is a finite or countable first order signature,
and that $T$ is a complete theory for $L$ whose models have at least two elements.

The \emph{randomization signature} $L^R$ is the two-sorted continuous signature
with sorts $\BK$ (for random elements) and $\BB$ (for events), an $n$-ary
function symbol $\l\varphi(\cdot)\rr$ of sort $\BK^n\to\BB$
for each first order formula $\varphi$ of $L$ with $n$ free variables,
a $[0,1]$-valued unary predicate symbol $\mu$ of sort $\BB$ for probability, and
the Boolean operations $\top,\bot,\sqcap, \sqcup,\neg$ of sort $\BB$.  The signature
$L^R$ also has distance predicates $d_\BB$ of sort $\BB$ and $d_\BK$ of sort $\BK$.
In $L^R$, we use ${\sa B},{\sa C},\ldots$ for variables or parameters of sort $\BB$. ${\sa B}\doteq{\sa C}$
means $d_\BB({\sa B},{\sa C})=0$, and ${\sa B}\sqsubseteq{\sa C}$ means ${\sa B}\doteq{\sa B}\sqcap{\sa C}$.

A pre-structure for $T^R$ will be a pair $\cu P=(\cu K,\cu B)$ where $\cu K$ is the part of sort $\BK$ and
$\cu B$ is the part of sort $\BB$.  The \emph{reduction} of $\cu P$ is the pre-structure $\cu N=(\hat{\cu K},\hat{\cu B})$ obtained from
$\cu P$ by identifying elements at distance zero, and the associated mapping from $\cu P$ onto $\cu N$ is called the \emph{reduction map}.
The \emph{completion} of $\cu P$ is the structure obtained by completing the metrics in the reduction of $\cu P$.
A pre-structure $\cu P$ is called \emph{pre-complete} if the reduction of $\cu P$ is already the completion of $\cu P$.

In [BK], the randomization theory $T^R$ is defined by listing a set of axioms.
We will not repeat these axioms here, because it is simpler to give the following model-theoretic
characterization of $T^R$.

\begin{df}  \label{d-nice}
Given a model $\cu M$ of $T$, a \emph{nice randomization of} $\cu M$ is a pre-complete structure $(\cu K,\cu B)$ for $L^R$
equipped with an atomless probability space $(\Omega,\cu B,\mu )$ such that:

\begin{enumerate}
\item $\cu B$ is a $\sigma$-algebra with $\top,\bot,\sqcap, \sqcup,\neg$ interpreted by $\Omega,\emptyset,\cap,\cup,\setminus$.
\item $\cu K$ is a set of functions $a\colon\Omega\to M$.
\item For each formula $\psi(\vec{x})$ of $L$ and tuple
$\vec{a}$ in $\cu K$, we have
$$\l\psi(\vec{a})\rr=\{\omega\in\Omega:\cu M\models\psi(\vec{a}(\omega))\}\in\cu B.$$
\item $\cu B$ is equal to the set of all events
$ \l\psi(\vec{a})\rr$
where $\psi(\vec{v})$ is a formula of $L$ and $\vec{a}$ is a tuple in $\cu K$.
\item For each formula $\theta(u, \vec{v})$
of $L$ and tuple $\vec{b}$ in $\cu K$, there exists $a\in\cu K$ such that
$$ \l \theta(a,\vec{b})\rr=\l(\exists u\,\theta)(\vec{b})\rr.$$
\item On $\cu K$, the distance predicate $d_\BK$ defines the pseudo-metric
$$d_\BK(a,b)= \mu \l a\neq b\rr .$$
\item On $\cu B$, the distance predicate $d_\BB$ defines the pseudo-metric
$$d_\BB({\sa B},{\sa C})=\mu ( {\sa B}\triangle {\sa C}).$$
\end{enumerate}
\end{df}

\begin{df}  For each first order theory $T$, the \emph{randomization theory} $T^R$ is
the set of sentences that are true in all nice randomizations of models of $T$.
\end{df}

It follows that for each first order sentence $\varphi$, if $T\models\varphi$
then $T^R\models \l\varphi\rr\doteq \top$.
The following basic facts are from [BK], Theorem 2.1 and Proposition 2.2, Example 3.4 (ii), Proposition 2.7, and Theorem 2.9.

\begin{fact}  \label{f-complete}
For every complete first order theory $T$, the randomization theory $T^R$ is complete.
\end{fact}

\begin{fact}  \label{f-T^R}
Every model $\cu M$ of $T$ has nice randomizations.
\end{fact}

\begin{fact}  \label{f-perfectwitnesses}  (Fullness)
Every pre-complete model $\cu P=(\cu K,\cu B)$ of $T^R$ has perfect witnesses, i.e.,
\begin{enumerate}
\item  For each first order formula $\theta(u,\vec v)$ and each $\vec{b }$ in $\cu K^n$ there exists $a \in\cu K$ such that
$$ \l\theta(a,\vec b)\rr \doteq
\l(\exists u\,\theta)(\vec{b })\rr;$$
\item For each ${\sa B}\in\cu B$ there exist $a ,b \in\cu K$ such that
${\sa B}\doteq\l a=b \rr$.
\end{enumerate}
\end{fact}

\begin{cor}  \label{c-two}  Every model $\cu N$ of $T^R$ has a pair of elements $ c,  d$ such that $\l  c\ne d\rr=\top$.
\end{cor}

\begin{proof}  Every model of $T$ has at least two elements, so $T\models(\exists u)(\exists v)u\ne v$.
The result follows by applying Fullness twice.
\end{proof}

\begin{fact} \label{f-qe} (Strong quantifier elimination)
Every formula $\Phi$ in the continuous language $L^R$
is $T^R$-equivalent to a formula with the same free variables
and no quantifiers of sort $\BK$ or $\BB$.
\end{fact}

\begin{lemma}  \label{l-glue}
Let $\cu P=(\cu K,\cu B)$ be a pre-complete model of $T^R$ and let $a ,b \in\cu K$ and ${\sa B}\in\cu B$.
Then there is an element $c \in\cu K$ that agrees with $a $ on ${\sa B}$ and agrees with $b $ on $\neg{\sa B}$,
that is, ${\sa B}\sqsubseteq\l c =a \rr$ and $(\neg{\sa B})\sqsubseteq\l c =b \rr$.
\end{lemma}

\begin{df}  In Lemma \ref{l-glue}, we will call $c$ a \emph{characteristic function of $\sa B$ with respect to $a,b$}.
\end{df}

Note that the distance between any two characteristic functions of an event $\sa B$ with respect to elements $a,b$ is zero.
In particular, in a model of $T^R$, the characteristic function is unique.

\begin{proof}[Proof of Lemma \ref{l-glue}]
By Fact \ref{f-perfectwitnesses} (2), there exist $d,e\in\cu K$ such that ${\sa B}\doteq\l d=e\rr$.  The first order sentence
$$(\forall u)(\forall v)(\forall x)(\forall y)(\exists z)[(x=y\rightarrow z=u)\wedge( x\neq y\rightarrow z=v)]$$
is logically valid, so we must have
$$\l (\exists z)[(d=e\rightarrow z=a)\wedge (d\ne e\rightarrow z=b)]\rr\doteq \top.$$
By Fact \ref{f-perfectwitnesses} (1) there exists $c\in\cu K$ such that
$$\l d=e\rightarrow c=a\rr\doteq\top,\quad \l d\ne e\rightarrow c=b\rr\doteq \top,$$
so $\l d=e\rr\sqsubseteq \l c=a\rr$ and $\l d\ne e\rr\sqsubseteq\l c=b\rr$.
\end{proof}

We will need the following result, which is a consequence of Theorem 3.11 of [Be].  Since the setting in [Be] is quite different
from the present paper, we give a direct proof here.

\begin{prop}  \label{p-representation}  Every model of $T^R$ is isomorphic to the reduction of a nice randomization of a model of $T$.
\end{prop}

\begin{proof}  Let $\cu N=(\hat{\cu K},\hat{\cu B})$ be a model of $T^R$ of cardinality $\kappa$.
Let $\Omega$ be the Stone space of the Boolean algebra $\hat{\cu B}=(\hat{\cu B},\top,\bot,\sqcap, \sqcup,\neg)$.  Thus $\Omega$ is a compact topological
space, the points of $\Omega$ are ultrafilters, we may identify $\hat{\cu B}$ with the Boolean algebra of clopen sets of $\Omega$, and $\mu^\cu N$ is a
finitely additive probability measure on $\hat{\cu B}$.

We next show that $\mu$ is $\sigma$-additive on $\hat{\cu B}$.  To do this, we assume that ${\sa A}_0\supseteq{\sa A}_1\supseteq\cdots$ in $\hat{\cu B}$
and ${\sa C} =\bigcap_n{\sa A}_n\in\hat{\cu B}$, and prove that $\mu(\sa C)=\lim_{n\to\infty}\mu({\sa A}_n)$.
Indeed, the family $\{\sa C\cup(\Omega\setminus{\sa A}_n)\colon n\in\BN\}$ is an open covering of $\Omega$, so by the topological compactness of $\Omega$,
we have $\Omega=\bigcup_{k=0}^n (\sa C\cup(\Omega\setminus{\sa A}_k))$ for some $n\in\BN$.  Then $\sa C={\sa A}_n$, so
$\mu(\sa C)=\mu({\sa A}_n)=\lim_{n\to\infty}\mu({\sa A}_n)$.

By the Caratheodory theorem,
there is a complete probability space $(\Omega,\cu B,\mu)$ such that $\cu B\supseteq\hat{\cu B}$,
$\mu$ agrees with $\mu^\cu N$ on $\hat{\cu B}$, and for each ${\sa B}\in\cu B$ and $m>0$ there is a countable sequence
${\sa A}_{m0}\subseteq {\sa A}_{m1}\subseteq\cdots$ in $\hat{\cu B}$ such that
\begin{equation}  \label{eq-rep}
B\subseteq \bigcup_n {\sa A}_{mn} \mbox{ and } \mu\left(\bigcup_n {\sa A}_{mn}\right)\le\mu({\sa B})+1/m.
\end{equation}
Note that since the probability space $(\Omega,\cu B,\mu)$ is complete, every subset of $\Omega$ that contains a set in $\cu B$
of measure one also belongs to $\cu B$ and has measure one.

We claim that for each $\sa B\in\cu B$ there is a unique event $f(\sa B)\in\hat{\cu B}$ such that $\mu(f(\sa B)\triangle{\sa B})=0$.
The uniqueness of $f(\sa B)$ follows from the fact that the distance function
$d_\BB(\sa C,\sa D)=\mu(\sa C\triangle\sa D)$ is a metric on $\hat{\cu B}$.   To show the existence of $f(\sa B)$, for each
$m>0$ let ${\sa A}_{m0}\subseteq {\sa A}_{m1}\subseteq\cdots$  be as in (\ref{eq-rep}).  Note that
$({\sa A}_{m0},{\sa A}_{m1},\ldots)$ is a Cauchy sequence of events in the model $\cu N$, so there is an event ${\sa C}_m\in\hat{\cu B}$ such that
${\sa C}_m=\lim_{n\to\infty} {\sa A}_{mn}$.  Hence $\lim_{n\to\infty}\mu({\sa A}_{mn}\triangle {\sa C}_m)=0$, so
$\mu((\bigcup_n {\sa A}_{mn})\triangle {\sa C}_m)=0$.
Then $({\sa C}_1,{\sa C}_2,\ldots)$ is a Cauchy sequence, so there is an event $f(\sa B)=\lim_{m\to\infty} {\sa C}_m$ in $\hat{\cu B}$
with $\mu(f(\sa B)\triangle{\sa B})=0$.

We make some observations about the mapping $f\colon\cu B\to\hat{\cu B}$. If $\sa B, \sa C\in\cu B$ and $d_\BB(\sa B,\sa C)=0$, then $f(\sa B)=f(\sa C)$.
For each $\sa B, \sa C\in\cu B$, we have
$$f(\sa B\cup\sa C)=f(\sa B)\cup f(\sa C),\qquad  f(\sa B\cap\sa C)=f(\sa B)\cap f(\sa C),$$
$$\Omega\setminus f(\sa B)=f(\Omega\setminus\sa B), \qquad \mu(\sa B)=\mu(f(\sa B)).$$
Moreover, the mapping $f$ sends $\cu B$ onto $\hat{\cu B}$, because if $\sa C\in\hat{\cu B}$ then $\sa C\in\cu B$ and $f(\sa C)=\sa C$.
Therefore the mapping $\hat f$ that sends the equivalence class of each $\sa B\in\cu B$ under $d_\BB$ to $f(\sa B)$ is well defined and
is an isomorphism from the reduction of the pre-structure $(\cu B,\sqcup,\sqcap,\neg.\top,\bot,\mu)$ onto the measured algebra
$(\hat{\cu B},\sqcup,\sqcap,\neg.\top,\bot,\mu)$.

A model $\cu M$ of $T$ is
\emph{$\kappa^+$-universal} if every model of $T$ of cardinality $\le\kappa$ is elementarily embeddable in $\cu M$.  By Theorem 5.1.12 in [CK],
every $\kappa$-saturated model of $T$ is $\kappa^+$-universal, so $\kappa^+$-universal models  of $T$ exist.
We now assume that $\cu M$ is a $\kappa^+$-universal model of $T$, and prove that $\cu N$ is isomorphic to the reduction of
a nice randomization of $\cu M$ with the underlying probability space $(\Omega,\cu B,\mu)$.

In the following paragraphs, we will use boldface letters $\bo b,\bo d,\ldots$ for elements of $\hat{\cu K}$.
Let $L_{\hat{\cu K}}$ be the first order signature formed by adding a constant symbol for each element $\bo b\in\hat{\cu K}$.
For each $\omega\in\Omega$, the set of $L_{\hat{\cu K}}$-sentences
$$ U(\omega)=\{\psi(\vec{\bo b})\colon \omega\in\l\psi(\vec{\bo b})\rr\}$$
is consistent with $T$ and has cardinality $\le\kappa$. By the  Compactness and L\"{o}wenheim-Skolem theorems, each $U(\omega)$ has a model
$(\cu M_\omega,{\bo b}_\omega)_{\bo b\in\hat{\cu K}}$ of cardinality $\le\kappa$.  Since $\cu M$ is $\kappa^+$-universal,
for each $\omega\in\Omega$ we may choose an
elementary embedding $h_\omega\colon\cu M_\omega\prec\cu M$. Then $(\cu M,h_\omega({\bo b}_\omega))_{\bo b\in\hat{\cu K}}\models U(\omega)$ for every $\omega\in\Omega$.  It follows that for each formula $\psi(\vec v)$ of $L$ and each tuple $\vec {\bo b}\in\hat{\cu K}^{<\BN}$,
$$ \l\psi(\vec{\bo b})\rr=\{\omega\in\Omega\colon\cu M_\omega\models\psi(\vec {\bo b}_\omega)\}=
\{\omega\in\Omega\colon \cu M\models\psi(h_\omega(\vec {\bo b}_\omega))\}\in\hat{\cu B}.$$
For each formula $\psi(\vec v)$ of $L$ and tuple $\vec c$ of functions in $M^\Omega$, define
$$ \l\psi(\vec c)\rr:=\{\omega\in\Omega\colon \cu M\models\psi(\vec c(\omega))\}.$$

Let $\cu K$ be the set of all functions $a\colon\Omega\to M$ such that for some element $\bo b\in\hat{\cu K}$, we have
$$\mu(\{\omega\in\Omega\colon a(\omega)=h_\omega({\bo b}_\omega)\})=1.$$
We claim that for each $a\in\cu K$ there is a unique element $f(a)\in\hat{\cu K}$ such that
$$\mu(\{\omega\in\Omega\colon a(\omega)=h_\omega(f(a)_\omega)\})=1.$$
The existence of $f(a)$ is guaranteed by the definition of $\cu K$. To prove uniqueness, suppose $\bo b, \bo d\in\hat{\cu K}$ and
$$\mu(\{\omega\in\Omega\colon a(\omega)=h_\omega({\bo b}_\omega)\})=\mu(\{\omega\in\Omega\colon a(\omega)=h_\omega({\bo d}_\omega)\})=1.$$
Then
$$\mu(\{\omega\in\Omega\colon h_\omega({\bo b}_\omega)=h_\omega({\bo d}_\omega) \})=1,$$
so
$$\mu(\l\bo b=\bo d\rr)=\mu(\{\omega\in\Omega\colon {\bo b}_\omega={\bo d}_\omega\})=1,$$
and hence $d_\BK(\bo b,\bo d)=0$.  Since $d_\BK$ is a metric on $\hat{\cu K}$, it follows that $\bo b=\bo d$.

We now make some observations about the mapping $f\colon\cu K\to\hat{\cu K}$.  This mapping sends $\cu K$ onto $\hat{\cu K}$, because for each $\bo b\in\hat{\cu K}$,
we have $f(a)=\bo b$ where $a$ is the element of $\cu K$ such that $a(\omega)=h_\omega({\bo b}_\omega)$ for all $\omega\in\Omega$.
Suppose $\vec c\in{\cu K}^{<\BN}$ and $\vec {\bo d}=f(\vec c)$.  We have $\vec {\bo d}\in\hat{\cu K}^{<\BN}$ and
$$\l\psi(\vec {\bo d})\rr=\{\omega\in\Omega\colon \cu M\models\psi(h_\omega(\vec {\bo d}_\omega))\}\doteq
\{\omega\in\Omega\colon \cu M\models\psi(\vec c(\omega))\}=\l\psi(\vec c)\rr.$$
Since the probability space $(\Omega,\cu B,\mu)$ is complete, $\l\psi(\vec {\bo d})\rr\in\hat{\cu B}\subseteq\cu B$, and
$\l\psi(\vec {\bo d})\rr\doteq\l\psi(\vec c)\rr$, we have  $\l\psi(\vec c)\rr\in\cu B$ and $\l\psi(\vec {\bo d})\rr=f(\l\psi(\vec c)\rr)$.
Therefore, if $a,c\in\cu K$ and $d_\BK(a,c)=0$, then $d_\BK(f(a),f(c))=0$, and hence $f(a)=f(c)$.
This shows that $\cu P=(\cu K,\cu B)$ is a well-defined pre-complete structure for $L^R$, and that the mapping $\hat f$ that sends the equivalence class of
each ${\sa B}\in\cu B$ to $f(\sa B)$, and the equivalence class of each $a\in\cu K$ to $f(a)$, is an isomorphism from the reduction of $\cu P$ to $\cu N$.

It remains to show that $\cu P$ is a nice randomization of $\cu M$.  It is clear that $\cu P$ satisfies conditions (1)-(3) in
Definition \ref{d-nice}.

Proof of (4):   We have already shown that $\l\psi(\vec c)\rr\in\cu B$
for each  formula $\psi(\vec v)$ of $L$ and each tuple $\vec c$ in $\cu K$.  For the other direction, let $\sa B\in\cu B$.
By Corollary \ref{c-two}, there exist $a,e\in\cu K$ such that
$\l a\ne e\rr\doteq\Omega$.  We may choose a function $b\in M^\Omega$ such that $b(\omega)=e(\omega)$ whenever $a(\omega)\ne e(\omega)$,
and $b(\omega)\ne a(\omega)$ for all $\omega\in\Omega$.  Then $b\in\cu K$ and $\l a\ne b\rr=\Omega$.
By Lemma \ref{l-glue}, there exists $c\in\cu K$ which is a characteristic function of $\sa B$ with
respect to $a,b$.  Then $\l c=a\rr\doteq\sa B$.  Let $d\in M^\Omega$ be the function such that $d(\omega)=a(\omega)$ for $\omega\in\sa B$,
and $d(\omega)=b(\omega)$ for $\omega\in\neg\sa B$.  Then $\mu(\l c=d\rr)=1$, so $d\in\cu K$, and $\l a=d\rr=\sa B$.  Thus (4) holds with
$\psi$ being the sentence $a=d$.

Proof of (5):  Consider a formula $\theta(u,\vec v)$ of $L$ and a tuple $\vec b$ in $\cu K$.  By Fullness,
there exists $c\in\cu K$ such that
$$\l\theta(c,\vec b)\rr\doteq\l(\exists u)\theta(u,\vec b)\rr.$$
We may choose a function $a\in M^\Omega$ such that for all $\omega\in\Omega$,
$$\cu M\models [\theta(c(\omega),\vec b(\omega))\leftrightarrow (\exists u)\theta(u,\vec b)] \mbox{ implies } a(\omega)=c(\omega),$$
and
$$\cu M\models [(\exists u)\theta(u,\vec b(\omega))\rightarrow\theta(a(\omega),\vec b(\omega))].$$
Then $\mu(\l a=c\rr)=1$, so $a\in\cu K$ and
$$\l\theta(a,\vec b)\rr=\l(\exists u)\theta(u,\vec b)\rr,$$
as required.

Proof of (6) and (7):  By Fact \ref{f-T^R}, the properties
$$ (\forall x)(\forall y) d_\BK(x,y)=\mu(\l x\ne y\rr),\quad (\forall \sa U)(\forall \sa V)d_\BB(\sa U,\sa V)=\mu(\sa U\triangle\sa V)$$
hold in some model of $T^R$.  By Fact \ref{f-complete}, these properties hold in all models of $T^R$, and thus in $\cu N$.
Therefore (6) and (7) hold for $\cu P$.
\end{proof}

\subsection{Types and Definability}

For a first order structure $\cu M$ and a set $A$ of elements of $\cu M$, $\cu M_A$ denotes the structure formed by
adding a new constant symbol to $\cu M$ for each $a\in A$.
The \emph{type realized by} a tuple $\vec b$ over the parameter set $A$ in $\cu M$ is the set $\tp^\cu M(\vec b/A)$ of formulas
$\varphi(\vec u,{\vec a})$
with $\vec a\in A^{<\BN}$ satisfied by $\vec b$ in $\cu M_A$.  We call $\tp^{\cu M}(\vec b/A)$ an \emph{$n$-type} if $n=|\vec b|$.

In the following, let $\cu N$ be a continuous structure and let $ A$ be a set of elements of $\cu N$.  $\cu N_{ A}$ denotes the structure formed by
adding a new constant symbol to $\cu N$ for each $ a\in  A$.
The \emph{type} $\tp^\cu N(\vec{ b}/ A)$ \emph{realized} by $\vec { b}$ over the parameter set $ A$ in $\cu N$
is the function $p$ from formulas to $[0,1]$ such that
for each formula $\Phi(\vec{x},\vec { a})$ with $\vec{ a}\in { A}^{<\BN}$, we have
$\Phi(\vec{x},\vec { a})^p=\Phi(\vec { b},\vec { a})^\cu N$.

We now recall the notions of definable element and algebraic element from [BBHU].
An element ${ b}$ is \emph{definable over} $ A$ in $\cu N$, in symbols $ b\in\dcl^\cu N( A)$,
if there is a sequence of formulas $\langle\Phi_k(x,\vec{ a}_k)\rangle$ with $\vec{ a}_k\in{ A}^{<\BN}$ such
that the sequence of functions $\langle\Phi_k(x,\vec{ a}_k)^\cu N\rangle$ converges uniformly in $x$ to the distance function $d(x, b)^\cu N$
of the corresponding sort.
$ b$ is \emph{algebraic over $ A$} in $\cu N$, in symbols $ b\in\acl^\cu N( A)$, if there is a compact set $C$ and
a sequence of formulas $\langle\Phi_k(x,\vec { a}_k)\rangle$ with $\vec{ a}_k\in{ A}^{<\BN}$ such
that $b\in C$ and the sequence of functions $\langle\Phi_k(x,\vec{ a}_k)^\cu N\rangle$  converges uniformly in $x$ to the distance function $d(x,C)^\cu N$
of the corresponding sort.

If the structure $\cu N$ is clear from the context, we will sometimes drop the superscript and write $\tp, \dcl, \acl$ instead of $\tp^\cu N, \dcl^\cu N, \acl^\cu N$.

\begin{fact}  \label{f-definable} ([BBHU], Exercises 10.7 and 10.10)  For each element $ b$ of $\cu N$,
the following are equivalent, where $p=\tp^\cu N( b/ A)$:
\begin{enumerate}
\item $ b$ is definable over $ A$ in $\cu N$;
\item in each model $\cu N'\succ\cu N$, $ b$ is the a unique element that realizes $p$ over $ A$;
\item $ b$ is definable over some countable subset of $ A$ in $\cu N$.
\end{enumerate}
\end{fact}

\begin{fact}  \label{f-algebraic} ([BBHU], Exercise 10.8 and 10.11)
For each element $ b$ of $\cu N$,
the following are equivalent, where $p=\tp^\cu N( b/ A)$:
\begin{enumerate}
\item $ b$ is algebraic over $ A$ in $\cu N$;
\item in each model $\cu N'\succ\cu N$, the set of elements $ b$ that realize $p$ over $ A$ in $\cu N'$ is compact.
\item $ b$ is algebraic over some countable subset of $ A$ in $\cu N$.
\end{enumerate}
\end{fact}

\begin{fact} \label{f-definableclosure}  (Definable Closure, Exercises 10.10 and 10.11 in [BBHU])
\begin{enumerate}
\item If $ A\subseteq\cu N$ then
$\dcl( A)=\dcl(\dcl( A))$ and $\acl( A)=\acl(\acl( A))$.
\item If $ A$ is a dense subset of $ B$ and $ B\subseteq\cu N$, then $\dcl(A)=\dcl( B)$ and $\acl(A)=\acl( B)$.
\end{enumerate}
\end{fact}

It follows that for any $ A\subseteq\cu N$, $\dcl( A)$ and $\acl( A)$ are closed with respect to the metric in $\cu N$.

We now turn to the case where $\cu N$ is a model of $T^R$. In that case, a set of elements of $\cu N$ may contain elements of both sorts $\BK, \BB$.
But as we will now explain, we need only consider definability over sets of parameters of sort $\BK$.

\begin{rmk}  \label{r-sortK-definability}  Let $\cu N=(\hat{\cu K},\hat{\cu B})$ be a model of $T^R$.  Since every model of $T$ has at
least two elements, $\cu N$ has a pair of elements $ a,  b$ of sort $\BK$ such that $\cu N\models\l a= b\rr=\bot$.
For each event ${\sa D}\in\hat{\cu B}$, let $1_{\sa D}$ be the characteristic function of ${\sa D}$ with respect to $ a,  b$.
Then in the model $\cu N$,
${\sa D}$ is definable over $\{ a, b,1_{\sa D}\}$, and $1_{\sa D}$ is definable over  $\{ a, b,{\sa D}\}$.
\end{rmk}

\begin{proof}  By Fact \ref{f-definable}.
\end{proof}

In view of Remark \ref{r-sortK-definability} and Fact \ref{f-definableclosure}, if $C$ is a set of parameters in $\cu N$
of both sorts, and there are elements $a,b\in C$ such that $\cu N\models\l a= b\rr=\bot$, then an element of either sort is definable
over $C$ if and only if it is definable over the set of parameters
of sort $\BK$ obtained by replacing each element of $C$ of sort $\BB$ by its characteristic
function with respect to $ a, b$.  For this reason, in a model $\cu N$ of $T^R$
we will only consider definability over sets of parameters of sort $\BK$.
We write $\dcl_\BB( A)$ for the set of elements of sort $\BB$ that are definable over $ A$ in $\cu N$,
and write $\dcl( A)$ for the set of elements of sort $\BK$ that are definable over $ A$ in $\cu N$.
Similarly for $\acl_\BB( A)$ and $\acl( A)$.

\subsection{Conventions and Notation}  We will assume hereafter that
$\cu N=(\hat{\cu K},\hat{\cu B})$ is a model of $T^R$, $\cu P=(\cu K,\cu B)$ is a nice randomization of a model $\cu M\models T$
with probability space $(\Omega,\cu B,\mu)$, and $\cu N$ is the reduction of $\cu P$.  The existence of $\cu P$ is guaranteed by Proposition \ref{p-representation}.

We will use boldfaced letters $\bo a,\bo b,\ldots$ for elements of $\hat{\cu K}$.
For each element $\bo a\in\hat{\cu K}$, we will choose once and for all an element $a\in\cu K$ such that
the image of $a$ under the reduction map is $\bo a$.  It follows that for each first order formula $\varphi(\vec v)$,
$\l\varphi(\vec{\bo a})\rr$ is the image of $\l\varphi(\vec a)\rr$ under the reduction map.
For any countable set $ A\subseteq\hat{\cu K}$ and each $\omega\in\Omega$, we define
$$ A(\omega)=\{a(\omega)\colon \bo a\in  A\}.$$
When $ A\subseteq\hat{\cu K}$, $\cl( A)$ denotes the closure of $ A$ in the metric $d_\BK$.  When $ B\subseteq\hat{\cu B}$,
$\cl( B)$ denotes the closure of $ B$ in the metric $d_\BB$, and
$\sigma( B)$ denotes the smallest $\sigma$-subalgebra of $\hat{\cu B}$ containing $ B$.

\section{Randomizations of Arbitrary Theories}  \label{s-arb}

\subsection{Definability in Sort $\BB$}

We  characterize the set of elements of $\hat{\cu B}$ that are definable in $\cu N$ over a set of parameters $ A\subseteq\hat{\cu K}$.

\begin{df} For each $A\subseteq \hat{\cu K}$, we say that an event ${\sa E}$ is \emph{first order definable} over $A$, in symbols
${\sa E}\in\fo_\BB(A)$,
if $\sa E=\l\varphi(\vec{\bo a})\rr$ for some first order formula $\varphi(\vec v)$ and tuple $\vec{\bo a}$ in ${A}^{<\BN}$.
\end{df}

\begin{thm}  \label{t-definableB}
For each $ A\subseteq \hat{\cu K}$,  $\dcl_\BB( A)=\cl(\fo_\BB( A))=\sigma(\fo_\BB( A))$.
\end{thm}

\begin{proof}
By quantifier elimination (Fact \ref{f-qe}), in any elementary extension $\cu N'\succ\cu N$, two events have the same type over
$ A$ if and only if they have the same type over $\fo_\BB( A)$.  Then by Fact \ref{f-definable},
$\dcl_\BB( A)=\dcl_\BB(\fo_\BB( A))$.
Moreover, $\dcl_\BB(\fo_\BB( A))$
is equal to the definable closure of $\fo_\BB( A)$ in the pure measured algebra $(\hat{\cu B},\mu)$.  By Observation 16.7 in [BBHU], the definable closure
of $\fo_\BB( A)$ in $(\hat{\cu B},\mu)$ is equal to $\sigma(\fo_\BB( A))$, so $\dcl_\BB( A)=\sigma(\fo_\BB( A))$.  Since $\fo_\BB( A)$ is a Boolean subalgebra
of $\hat{\cu B}$, $\cl(\fo_\BB( A))$ is a Boolean subalgebra of $\hat{\cu B}$.  By metric completeness, $\cl(\fo_\BB( A))$ is a $\sigma$-algebra and
$\sigma(\fo_\BB( A))$ is closed, so $\cl(\fo_\BB( A))=\sigma(\fo_\BB( A))$.
\end{proof}

\begin{cor}  \label{c-event-noparameters}
The only events that are definable without parameters in $\cu N$ are $\top$ and $\bot$.
\end{cor}

\begin{proof}  For every first order sentence $\varphi$, either $T\models\varphi$ and $T^R\models\l\varphi\rr=\top$,
or $T\models\neg\varphi$ and $T^R\models\l\varphi\rr=\bot$.  So  $\fo_\BB(\emptyset)=\{\top,\bot\}$.
\end{proof}

\subsection{First Order and Pointwise Definability}

To prepare the way for a characterization of the definable elements of sort $\BK$,
we introduce two auxiliary notions, one that is stronger than definability in sort $\BK$ and one that is weaker than definability in sort $\BK$.
We will work in the nice randomization $\cu P=(\cu K,\cu B)$  of $\cu M$, and let $A$ be a subset of
$\hat{\cu K}$ and $\bo b$ be an element of $\hat{\cu K}$.

\begin{df}  A first order formula $\varphi(u,\vec v)$ is \emph{functional} if
$$T\models(\forall \vec v)(\exists ^{\le  1} u)\varphi(u,\vec v).$$

We say that $\bo b$ is \emph{first order definable on ${\sa E}$ over $A$}  if there is a functional formula
$\varphi(u,\vec v)$ and a tuple $\vec{\bo a}\in {A}^{<\BN}$ such that
$\sa E=\l \varphi(\bo b,\vec{\bo a})\rr$.

We say that $\bo b$ is \emph{first order definable over $ A$}, in symbols $\bo b\in\fo( A)$, if $\bo b$ is first order definable on
$\top$ over $A$.
\end{df}

\begin{rmks}  \label{r-definableover}
$\bo b$ is first order definable over $ A$ if and only if there is a first order formula $\varphi(u,\vec v)$ and a tuple $\vec{\bo a}$
from $ A$ such that
$$\mu(\l(\forall u)(\varphi(u,\vec{\bo a})\leftrightarrow u=\bo b)\rr)=1.$$

First order definability has finite character, that is, $\bo b$ is first order definable over $ A$ if and only if
$\bo b$ is first order definable over some finite subset of $ A$.

If $\bo b$ is first order definable on ${\sa E}$ over $A$, then $\sa E$ is first order definable over $A\cup\{\bo b\}$.

If $\bo b$ is first order definable on ${\sa D}$ over $A$, and $\sa E$ is first order definable over $A\cup\{\bo b\}$, then
$\bo b$ is first order definable on ${\sa D}\sqcap\sa E$ over $A$.
\end{rmks}

\begin{lemma}  \label{l-firstorderdefinable}
If $\bo b$ is first order definable over $ A$ then $\bo b$ is definable over $ A$ in $\cu N$.  Thus $\fo( A)\subseteq\dcl( A)$.
\end{lemma}

\begin{proof}  Let $\cu N'\succ\cu N$ and suppose that $\tp^{\cu N'}(\bo b)=\tp^{\cu N'}(\bo d)$. Then
$$\l \varphi(\bo b,\vec{\bo a})\rr=\l \varphi(\bo d,\vec{\bo a})\rr=\top.$$
Since $\varphi$ is functional,
$$ \l(\forall t)(\forall u) (\varphi(t,\vec{\bo a})\wedge\varphi(u,\vec{\bo a})\rightarrow t=u)\rr=\top.$$
Then $\l \bo b = \bo d\rr=\top$, so $\bo b=\bo d$, and by Fact \ref{f-definable}, $\bo b\in\dcl( A)$.
\end{proof}

\begin{df}  When $A$ is countable, we define
$$\l b\in\dcl^{\cu M}(A)\rr:=\{\omega\in\Omega\colon b(\omega)\in \dcl^\cu M(A(\omega))\}.$$
\end{df}

\begin{lemma}  \label{l-pointwisemeasurable}
If $A$ is countable, then
$$\l b\in\dcl^{\cu M}(A)\rr= \bigcup\{\l\theta(b,\vec a)\rr\colon\theta(u,\vec v) \mbox{ functional, } \vec{\bo a}\in A^{<\BN}\},$$
and  $\l b\in\dcl^{\cu M}(A)\rr\in\cu B$.
\end{lemma}

\begin{proof}  Note that for every first order formula $\theta(u,\vec v)$, the formula
$$\theta(u,\vec v)\wedge(\exists^{\le 1}u)\,\theta(u,\vec v)$$
is functional.  Therefore $\omega\in\l b\in\dcl^{\cu M}(A)\rr$ if and only if
$b(\omega)\in \dcl^\cu M(A(\omega))$, and this holds if and only if there is a functional formula $\theta(u,\vec v)$
and a tuple $\vec {\bo a}\in A^{<\BN}$ such that  $\cu M\models \theta(b(\omega),\vec a(\omega)).$
Since $A$ and $L$ are countable, $\l b\in\dcl^{\cu M}(A)\rr$ is the union of countably many events in $\cu B$,
and thus belongs to $\cu B$.
\end{proof}

\begin{df}  When $A$ is countable, we say that $\bo b$ is \emph{pointwise definable over $A$} if
$$\mu(\l b\in\dcl^{\cu M}(A)\rr)=1.$$
\end{df}

\begin{cor}  \label{c-pointwisedefinable}  If $A$ is countable, then $\bo b$ is pointwise definable over $A$ if and only if there
is a function $f$ on $\Omega$ such that:
\begin{enumerate}
\item For each $\omega\in \Omega$, $f(\omega)$ is a pair $\<\theta_\omega(u,\vec v),\vec a_\omega\>$ where $\theta_\omega(u,\vec v)$ is functional
and $\vec a_\omega\in A^{|\vec v|}$;
\item $f$ is $\sigma(\fo_\BB(A))$-measurable (i.e., the inverse image of each point belongs to $\sigma(\fo_\BB(A))$);
\item $\cu M\models \theta_\omega(b(\omega),\vec a_\omega(\omega))$ for almost every $\omega\in\Omega$.
\end{enumerate}
\end{cor}

\begin{proof}  If $\omega\in\l b\in\dcl^{\cu M}(A)\rr$, let $f(\omega)$ be the first pair
$\<\theta_\omega,\vec a_\omega\>$ such that $\theta_\omega(u,\vec v)$ is
functional, $\vec a_\omega\in A^{|\vec v|}$, and $\cu M\models \theta_\omega(b(\omega),\vec a_\omega(\omega))$.  Otherwise
let $f(\omega)=\<\bot,\emptyset\>$.  The result then follows from Lemma \ref{l-pointwisemeasurable}.
\end{proof}

\begin{lemma}  \label{l-pointwisedefinable}
If $\bo b$ is definable over $ A$ in $\cu N$, then $\bo b$ is pointwise definable over some countable subset of $A$.
\end{lemma}

\begin{proof}  By Fact \ref{f-definable} (3), we may assume that $A$ is countable.
By Lemma \ref{l-pointwisemeasurable}, the measure $r:= \mu(\l  b\in\dcl^{\cu M}(A)\rr)$ exists.
Suppose $\bo b$ is not pointwise definable over $A$.  Then $r<1$.  For each finite collection $\chi_1(u,\vec v),\ldots,\chi_n(u,\vec v)$
of first order formulas, each tuple $\vec{\bo a}\in A^{<\BN}$, and each $\omega\in\Omega\setminus \l  b\in\dcl^{\cu M}(A)\rr$, the
sentence
$$ (\exists u)[u\ne b(\omega)\wedge\bigwedge_{i=1}^n [\chi_i(b(\omega),\vec a(\omega))\leftrightarrow \chi_i(u,\vec a(\omega))] $$
holds in $\cu M$, because $b(\omega)$ is not definable over $A(\omega)$.  Therefore in $\cu P$ we have
$$ \mu\l (\exists u)[u\ne b\wedge\bigwedge_{i=1}^n [\chi_i(b,\vec a)\leftrightarrow \chi_i(u,\vec a)]\rr\ge 1-r.$$
By condition \ref{d-nice} (5), there is an element $\bo d\in\hat{\cu K}$ such that
$$ \mu\l d\ne b\wedge\bigwedge_{i=1}^n [\chi_i(b,\vec a)\leftrightarrow \chi_i(d,\vec a)]\rr\ge 1-r.$$
It follows that
$ \mu(\l d\ne b\rr)\ge 1-r $, and $\l \chi_i(b,\vec a)\rr\doteq\l \chi_i(d,\vec a)\rr$ for each $i\le n$.
By compactness, in some elementary extension of $\cu N$ there is an element $\bo d$ such that
$\mu\l\bo d\ne\bo b\rr\ge 1-r$, and $\l\chi(\bo b,\vec{\bo a})\rr=\l\chi(\bo d,\vec{\bo a})\rr$ for each
first order formula $\chi(u,\vec v)$.  Then $\bo d\ne\bo b$, and by quantifier elimination, $\tp(\bo d/A)=\tp(\bo b/A)$.
Hence by Fact \ref{f-definable} (2), $\bo b\notin\dcl( A)$.
\end{proof}

The following example shows that the converse of Lemma \ref{l-pointwisedefinable} fails badly.

\begin{ex} Let $\cu M$ be a finite structure with a constant symbol for every element.  Then every element of $\cu K$ is
pointwise definable without parameters, but the only elements of $\hat{\cu K}$ that are definable without parameters are
the equivalence classes of constant functions $b\colon\Omega\to\cu M$.
\end{ex}

\subsection{Definability in Sort $\BK$}

We will now give necessary and sufficient conditions for an element of $\bo b\in \hat{\cu K}$ to be definable over a parameter set $ A\subseteq\hat K$
in $\cu N$.

\begin{thm}  \label{t-dcl}
 $\bo{b}$ is definable over $ A$ if and only if there exist
pairwise disjoint events $\{\sa E_n\colon n\in\BN\}$ such that $\sum_{n\in\BN}\mu( \sa E_n)=1$, and for each $n$, ${\sa E}_n$
is definable over $ A$, and $\bo b$ is first order definable on $\sa E_n$ over $A$.
\end{thm}

\begin{proof}
$(\Rightarrow)$:  Suppose $\bo{b}\in \dcl(A)$.   By Lemma \ref{l-pointwisedefinable}, $\bo b$ is pointwise definable over some
countable subset $A_0$ of $A$.  The set of all events $\sa C$ such that $\bo b$ is first order definable on $\sa C$ over $A_0$
is countable, and may be arranged in a list $\{\sa C_n\colon n\in\BN\}$.   Let $\sa E_0=\sa C_0$, and
$$\sa E_{n+1}=\sa C_{n+1}\sqcap\neg(\sa C_0\sqcup\cdots\sqcup\sa C_n).$$
The events $\sa E_n$ are pairwise disjoint, and for each $n$ we have
$$\sa E_0\sqcup\cdots\sqcup\sa E_n=\sa C_0\sqcup\cdots\sqcup\sa C_n.$$
 By Remarks \ref{r-definableover}, for each $n$, $\bo b$ is first order
definable on  $\sa E_n$ over $A$.  By Lemma \ref{l-pointwisemeasurable} and pointwise definability,
$$ \sum_{n\in\BN}\mu(\sa E_n)=\lim_{n\to\infty}\mu(\sa C_0\sqcup\cdots\sqcup\sa C_n)=\mu(\l\dcl^{\cu M}(A_0)\rr)=1.$$
By Remarks \ref{r-definableover},  ${\sa E}_n$ is definable over $ A\cup\{\bo b\}$, and since $\bo b$ is definable over $ A$,
${\sa E}_n$ is definable over $ A$ by Fact \ref{f-definableclosure}.

\

$(\Leftarrow)$:  Let $\sa E_n$ be as in the theorem.  For each $n$, we have $\sa E_n=\l\theta_n(\bo b,\vec{\bo a}_n)\rr$ for some functional formula $\theta_n$
and tuple $\vec {\bo a}_n\in A^{<\BN}$.
Since $\sa E_n$ is definable over $A$, by Theorem \ref{t-definableB} there is a sequence of formulas $\psi_k(\vec v)$ and tuples $\vec{\bo a_k}\in A^{<\BN}$
such that
$$\lim_{k\to\infty}d_\BB(\l\psi_k(\vec{\bo a}_k)\rr,\l\theta_n(\bo b,\vec{\bo a})\rr)=0.$$
 Suppose $\bo d$ has the same type over $ A$ as $\bo b$ in some elementary extension ${\cu N}'$ of ${\cu N}$.  Then
$$\lim_{k\to\infty}d_\BB(\l\psi_k(\vec{\bo a}_k)\rr,\l\theta_n(\bo d,\vec{\bo a})\rr)=0.$$
Hence
$$\l\theta_n(\bo d,\vec{\bo a}_n)\rr = \l\theta_n(\bo b,\vec{\bo a}_n)\rr={\sa E}_n$$
in ${\cu N}'$.  Since $\theta_n(u,\vec v)$ is functional, we have $\l\theta_n(\bo b,\vec{\bo a})\rr\sqsubseteq\l\bo d=\bo b\rr$
for each $n$.  Then
$$\mu(\l \bo d=\bo b\rr)\ge\sum_{n\in\BN}\mu({\sa E}_n)=1,$$
so $\bo d =\bo b$.  Then by Fact \ref{f-definable}, $\bo b\in\dcl( A)$.
\end{proof}

\begin{cor}  \label{c-definableK-noparameters}
An element $\bo b \in\hat{\cu K}$ is definable without parameters if and only if $\bo b$ is first order definable
without parameters.  Thus $\dcl(\emptyset)=\fo(\emptyset)$.
\end{cor}

\begin{proof}  $(\Rightarrow)$: Suppose $\bo b\in\dcl(\emptyset)$.  By Theorem \ref{t-dcl}, there is an event
$\sa E$ such that $\mu(\sa E)>0$, ${\sa E}$ is definable without parameters, and $\bo b$ is first order definable on $\sa E$
without parameters.  By Corollary \ref{c-event-noparameters} we have ${\sa E}=\top$, so $\bo b$ is first order definable without parameters.

$(\Leftarrow)$: By Lemma \ref{l-firstorderdefinable}.
\end{proof}

\begin{cor}  \label{c-definable-finite}
If $\fo_\BB( A)$ is finite, then $\dcl_\BB( A)=\fo_\BB( A)$ and $\dcl( A)=\fo( A)$.
\end{cor}

\begin{proof}  $\dcl_\BB( A)=\fo_\BB( A)$ follows from Theorem \ref{t-definableB}.  Lemma \ref{l-firstorderdefinable}
gives $\dcl( A)\supseteq\fo( A)$.  For the other inclusion, suppose $\bo b\in\dcl( A)$.  By Theorem \ref{t-dcl}, there is a finite
partition $\sa E_0,\ldots,\sa E_k$ of $\top$, a tuple $\vec {\bo a}\in A^{<\BN}$, and first order formulas $\psi_i(\vec v)$ such that
$\sa E_i=\l\psi_i(\vec{\bo a})\rr$ and $\bo b$ is first order definable on $\sa E_i$.  Then there are functional formulas $\varphi_i(u,\vec v)$
such that $\sa E_i\doteq\l\varphi_i(\bo b,\vec{\bo a})\rr$.  We may take the formulas $\psi_i(\vec v)$ to be pairwise inconsistent and such that
$T\models\bigvee_{i=0}^n\psi(\vec v)$.  Then $\bigwedge_{i=0}^n (\psi_i(\vec v)\rightarrow\varphi_i(u,\vec v))$ is a functional formula such that
$$\l\bigwedge_{i=0}^n (\psi_i(\vec{\bo a})\rightarrow\varphi_i(\bo b,\vec {\bo a}))\rr=\top,$$
so $\bo b$ is first order definable over $ A$.
\end{proof}

\begin{cor}  \label{c-dcl2}
$\bo{b}$ is definable over $ A$ if and only if:
\begin{enumerate}
\item $\bo b$ is pointwise definable over some countable subset of $A$;
\item for each functional formula $\varphi(u,\vec v)$ and tuple $\vec {\bo a}\in A^{<\BN}$,  $\l \varphi (\bo b, \vec{\bo a})\rr$
is definable over $ A$.
\end{enumerate}
\end{cor}

\begin{proof}
$(\Rightarrow)$:  Suppose $\bo b \in\dcl( A)$.  Then (1) holds by Lemma \ref{l-pointwisedefinable}.  $\l \varphi (\bo b, \vec{\bo a})\rr$ is obviously definable over $ A\cup\{\bo b\}$, so $\l \varphi (\bo b, \vec{\bo a})\rr$ is definable over $ A$ by Fact \ref{f-definableclosure},
and thus (2) holds.

$(\Leftarrow)$:  Assume conditions (1) and (2).  By (1) and Lemma \ref{l-pointwisemeasurable}, there is a sequence of
functional formulas $\theta_n(u,\vec v)$ and tuples $\vec{\bo a}_n\in A^{<\BN}$ such that
$$\l b\in\dcl^{\cu M}(A)\rr= \bigcup_{n\in\BN}\l\theta_n(b,\vec a_n)\rr\doteq\Omega.$$
Let $\sa E_n=\l\theta_n(\bo b,\vec{\bo a}_n)\rr$, so $\bo b$ is first order definable on $\sa E_n$ over $A$.  By Remark \ref{r-definableover},
we may take the $\sa E_n$ to be pairwise disjoint, and thus $\sum_{n\in\BN}\mu(\sa E_n)=1$.
By (2), ${\sa E}_n$ is definable over $ A$ for each $n$. Then by Theorem \ref{t-dcl}, $\bo b \in\dcl( A)$.
\end{proof}

\begin{cor}  \label{c-dcl3}   $\bo{b}$ is definable over $ A$ if and only if:
\begin{enumerate}
\item $b$ is pointwise definable over some countable subset of $A$;
\item $\fo_\BB( A\cup\{\bo b\})\subseteq\dcl_\BB( A)$.
\end{enumerate}
\end{cor}

\begin{thm}  \label{t-separable}
$\bo b$ is definable over $ A$ if and only if $\bo b=\lim_{m\to\infty} \bo b_m$, where each $\bo b_m$
is first-order definable over $ A$.  Thus $\dcl( A)=\cl(\fo( A))$.
\end{thm}

\begin{proof}
$(\Rightarrow)$: Suppose that $\bo b\in \dcl( A)$.  If $A$ is empty, then $\bo b$ is already first order definable from $ A$
by Corollary \ref{c-definableK-noparameters}.  Assume $A$ is not empty and let $\bo c\in A$.
Let $\{\sa E_n\colon n\in\BN\}$ be as in Theorem \ref{t-dcl}, and fix an $\varepsilon>0$.  Then for some $n$,
$\sum_{k=0}^n\mu( \sa E_k)>1-\varepsilon$.
For each $k$, ${\sa E}_k$ is definable over $ A$, so by Theorem \ref{t-definableB}, there is an event $\sa D_k\in\fo_\BB(A)$
such that $\mu(\sa D_k\triangle\sa E_k)<\varepsilon/n$.  Since the events $\sa E_k$ are pairwise disjoint, we may also take the events $\sa D_k$
to be pairwise disjoint.  We have $\sa E_k=\l\theta_k(\bo b,\vec{\bo a}_k)\rr$ for some functional $\theta_k(u,\vec v)$, so we may assume that
$\sa D_k$ has the additional properties that $\sa D_k\sqsubseteq\l(\exists ! u)\theta_k(u,\vec{\bo a}_k)\rr$, and that
$\sa D_k=\l\psi_k(\vec{\bo a}_k)\rr$ for some formula $\psi_k(\vec v)$.
Then there is a unique element $\bo d\in\hat{\cu K}$ such that
$$\begin{cases}
\cu M\models\theta_k(d(\omega),\vec a_k(\omega)) & \mbox{ if } k \le n \mbox{ and } \omega\in\l\psi_k(\vec{a}_k)\rr,\\
d(\omega) = c(\omega) & \mbox{ if } \omega\in\Omega\setminus\bigcup_{k=0}^n \l\psi_k(\vec{a}_k)\rr.
\end{cases}
$$
Then $\bo d$ is first order definable over $ A$, and $d_\BK(\bo b,\bo d)<\varepsilon$.

$(\Leftarrow)$:  This follows because first order definability implies definability (Lemma \ref{l-firstorderdefinable})
 and the set $\dcl( A)$ is metrically closed (Fact \ref{f-definableclosure} (2)).
\end{proof}

The following result was proved in [Be] by an indirect argument using Lascar types.  We give a simple  direct proof here.

\begin{prop}  For any model $\cu N=(\hat{\cu K},\hat{\cu B})$ of $T^R$ and set $ A\subseteq\hat{\cu K}$,
$\acl_\BB( A)=\dcl_\BB( A)$ and $\acl( A)=\dcl( A)$.
\end{prop}

\begin{proof}  By Facts \ref{f-definable} and \ref{f-algebraic}, we may assume $\cu N$ is $\aleph_1$-saturated and $ A$ is countable.
Suppose an event ${\sa E}\in\hat{\cu B}$ is not definable over $ A$.
By Fact \ref{f-definable} and $\aleph_1$-saturation there exists $\sa D\in\hat{\cu B}$ such that $\tp(\sa D/ A)=\tp(\sa E/ A)$ but
$d_\BB(\sa D,\sa E)>0$.  By $\aleph_1$-saturation again, there is a countable sequence of events $\<{\sa F}_n\colon n\in\BN\>$ in $\hat{\cu B}$
such that
$$\mu(\sa C\cap{\sa F}_n)=\mu(\sa C\setminus{\sa F}_n)=\mu(\sa C)/2$$
for each $n$ and each event $\sa C$ in the Boolean algebra generated by
$$\fo_\BB(A)\cup\{\sa D,\sa E\}\cup\{{\sa F}_k\colon k<n\}.$$
For each $n$, let
$${\sa D}_n=(\sa D\cap{\sa F}_n)\cup(\sa E\setminus{\sa F}_n).$$
Then for each $\sa C\in\fo_\BB(A)$ and $n\in\BN$, we have
$$\mu({\sa D}_n\cap\sa C)=\mu(\sa D\cap \sa C)/2 + \mu(\sa E\cap \sa C)/2 =\mu(\sa E\cap \sa C).$$
By quantifier elimination, $\tp(\sa D_n/ A)=\tp(\sa E/ A)$ for each $n\in\BN$.  Moreover, whenever $k< n$ we have
$${\sa D}_n\setminus{\sa D}_k=((\sa D\setminus{\sa D}_k)\cap{\sa F}_n)\cup((\sa E\setminus{\sa D}_k)\setminus{\sa F}_n),$$
so
$$\mu({\sa D}_n\setminus{\sa D}_k)=\mu(\sa D\setminus{\sa D}_k)/2+\mu(\sa E\setminus{\sa D}_k)/2.$$
Note that whenever $\tp(\sa D'/A)=\tp(\sa D''/A)$, we have $\mu(\sa D')=\mu(\sa D'')$, and hence
$$\mu(\sa D'\setminus\sa D'')=\mu(\sa D''\setminus\sa D')=d_\BB(\sa D',\sa D'')/2.$$
Therefore
$$d_\BB({\sa D}_n,{\sa D}_k)=d_\BB(\sa D,{\sa D}_k)/2 + d_\BB(\sa E,{\sa D}_k)/2\ge d_\BB(\sa D,\sa E)/2.$$
It follows that the set of realizations of $\tp(\sa E/A)$ is not compact, and $\sa E$ is not algebraic over $ A$.
This shows that $\acl_\BB( A)=\dcl_\BB( A)$.

Now suppose $\bo b\in \acl( A)\setminus \dcl( A)$.
There is an element $\bo c\in\hat{\cu K}$ such that $\tp(\bo b/ A)=\tp(\bo c/ A)$ but $d_\BK(\bo b,\bo c)>0$.
For each first order formula $\psi(u,\vec v)$ and $\vec{\bo a}\in A^{<\BN}$,
$\l\psi(\bo b,\vec{\bo a})\rr\in \acl_\BB(\{\bo b\}\cup A)\subseteq\acl_\BB(\acl(A))$.
By Fact \ref{f-definableclosure},
$\l\psi(\bo b,\vec{\bo a})\rr\in \acl_\BB(A)$.  By the preceding paragraph, $\l\psi(\bo b,\vec{\bo a})\rr\in \dcl_\BB( A)$.
Since $\tp(\bo b/ A)=\tp(\bo c/ A)$, we have $\tp(\l\psi(\bo b,\vec{\bo a})\rr/A)=\tp(\l\psi(\bo c,\vec{\bo a})\rr/A)$.
By Fact \ref{f-definable}, it follows that $\l\psi(\bo b,\vec{\bo a})\rr=\l\psi(\bo c,\vec{\bo a})\rr$ for every first order formula $\psi(u,\vec v)$.
Then $\tp(b(\omega)/A(\omega))=\tp(c(\omega)/A(\omega))$ for $\mu$-almost all $\omega$.  By $\aleph_1$-saturation, there are countably many independent
events $\sa D_n\in\hat{\cu B}$ such that $\sa D_n\sqsubseteq\l\bo b \ne\bo c\rr$ and $\mu(\sa D_n)=d_\BK(\bo b,\bo c)/2$.
Let $\bo c_n$ agree with $\bo c$ on $\sa D_n$ and agree with $\bo b$ elsewhere.
We have $\tp(\bo c_n/A)=\tp(\bo b/A)$ for every $n\in\BN$, and $d_\BK(\bo c_n,\bo c_k)=d_\BK(\bo b,\bo c)/2$ whenever $k< n$.  Thus the set
of realizations of $\tp(\bo b/A)$ is not compact, contradicting the fact that $\bo b\in \acl( A)$.
\end{proof}

\section{A Special Case:  $\aleph_0$-categorical theories}

\subsection{Definability and $\aleph_0$-Categoricity}

We use our preceding results to characterize $\aleph_0$-categorical theories in terms of definability in randomizations.

\begin{thm}  \label{t-categorical}
The following are equivalent:
\begin{enumerate}
\item $T$ is $\aleph_0$-categorical;
\item $\fo_\BB( A)$ is finite for every finite $ A$;
\item $\dcl_\BB( A)$ is finite for every finite $ A$;
\item $\fo_\BB(A)=\dcl_\BB(A)$ for every finite $A$;
\item $\fo( A)$ is finite for every finite $ A$;
\item $\dcl( A)$ is finite for every finite $ A$.
\item $\fo(A)=\dcl(A)$ for every finite $A$;
\end{enumerate}
\end{thm}

\begin{proof}
By the Ryll-Nardzewski Theorem (see [CK], Theorem 2.3.13), (1) is equivalent to

(0) For each $n$ there are only finitely many formulas in $n$ variables up to $T$-equivalence.

Assume (0) and let $A\subseteq\hat{\cu K}$ be finite.  Then (2) holds.  Moreover, there are only finitely many functional formulas
in $|A|+1$ variables, so (5) holds.  Then by Corollary \ref{c-definable-finite}, (3), (4), (6), and (7) hold.

Now assume that (0) fails.

\emph{Proof that (2) and (3) fail}:  For some $n$ there are infinitely many formulas in $n$ variables that are not $T$-equivalent.  Hence there is an
$n$-type $p$ in $T$ without parameters that is not isolated.  So there are formulas $\varphi_1(\vec v), \varphi_2(\vec v),\ldots$ in $p$
such that for each $k>0$, $T\models \varphi_{k+1}\rightarrow\varphi_k$ but the formula $\theta_k=\varphi_k\wedge\neg\varphi_{k+1}$ is consistent
with $T$.  The formulas $\theta_k$ are consistent but pairwise inconsistent.  By Fullness, for each $k>0$ there exists an $n$-tuple
$\vec {\bo b}_k\in\hat{\cu K}^n$
such that $\l\theta_k(\vec{\bo b}_k)\rr=\top$.  Since the measured algebra $(\hat{\cu B},\mu)$ is atomless, there are pairwise disjoint
events $\sa E_1,\sa E_2,\ldots$ in $\hat{\cu B}$ such that $\mu(\sa E_k)=2^{-k}$ for each $k>0$.  By applying Lemma \ref{l-glue} $k$ times, we
see that for each $k>0$ there is an $n$-tuple $\vec{\bo a}_k\in\hat{\cu K}^n$ that agrees with $\vec{\bo b}_i$ on $\sa E_i$ whenever $0<i\le k$.
Whenever $0<k\le j$, we have $\mu(\l \vec {\bo a}_k=\vec{\bo a}_j\rr)\ge 1-2^{-k}$.  So $\<\vec {\bo a}_1,\vec{\bo a}_2,\ldots\>$ is a
Cauchy sequence, and by metric completeness the limit $\vec {\bo a}=\lim_{k\to\infty}\vec{\bo a}_k$ exists in $\hat{\cu K}^n$.
Let $A=\range(\vec{\bo a})$.  For each $k>0$ we have $\sa E_k=\l\vec{\bo a}=\vec{\bo b}_k\rr=\l\theta_k(\vec{\bo a})\rr$, so $\sa E_k\in\fo_\BB(A)$.
Thus $\fo_\BB(A)$ is infinite, so (2) fails and (3) fails.

\emph{Proof that (4) fails}:  Let $\sa E_k$ be as in the preceding paragraph.
The set $\fo_\BB(A)$ is countable.  But the closure $\cl(\fo_\BB(A))$ is uncountable, because
for each set $S\subseteq\BN\setminus\{0\},$ the supremum $\bigsqcup_{k\in S}\sa E_k$ belongs to $\cl(\fo_\BB(A))$.  Thus by Theorem \ref{t-definableB},
$$\dcl_\BB(A)=\cl(\fo_\BB(A))\ne\fo_\BB(A),$$
and (4) fails.

\emph{Proof that (5), (6), and (7) fail}:  By Corollary \ref{c-two}, there exist $\bo c, \bo d\in\cu K$ such that $\l \bo c\ne\bo d\rr=\top$.
Let $C$ be the finite set $C=A\cup\{\bo c,\bo d\}$.  By Remark \ref{r-sortK-definability}, for any event $\sa D\in\fo_\BB( A)$, the
characteristic function $1_{\sa D}$ of $\sa D$ with respect to $\bo c,\bo d$ is definable over
$C$.  Moreover, we always have $d_\BK(1_{\sa D},1_{\sa E})=d_\BB(\sa D,\sa E)$.
It follows that $\fo(C)$ is infinite, so (5) and (6) fail.  To show that (7) fails, we take an event $\sa D\in \dcl_\BB(A)\setminus\fo_\BB(A)$.
By Theorem \ref{t-definableB} we have $\sa D\in\cl(\fo_\BB(A))$.  It follows that $1_{\sa D}\in\cl(\fo(C))$, so by Theorem \ref{t-separable},
$1_{\sa D}\in\dcl(C)$.  Hence $\dcl(C)$ is uncountable.  But $\fo(C)$ is countable, so (7) fails.
\end{proof}

By the Ryll-Nardzewski Theorem, if $T$ is $\aleph_0$-categorical then for each $n$, $T$ has finitely many $n$-types;
so each type $p$ in the variables $(u,\vec v)$ has an \emph{isolating formula}, that is, a formula $\varphi(u,\vec v)$ such that
$T\models \varphi(u,\vec v)\leftrightarrow \bigwedge p$.

We now characterize the definable closure of a finite set $ A\subseteq \hat{\cu K}$ in the case that $T$ is $\aleph_0$-categorical.
Hereafter, when $A$ is a finite subset of $\hat{\cu K}$, $\vec{\bo a}$ will denote a finite tuple whose range is $A$.

\begin{cor}  \label{c-cat-definable1}
Suppose that $T$ is $\aleph_0$-categorical, $\bo b\in\hat{\cu K}$, and $A$ is a finite subset of $\hat{\cu K}$.  Then
$\bo{b}\in \dcl( A)$ if and only if:
\begin{enumerate}
\item $\bo b$ is pointwise definable over $A$;
\item for every isolating formula $\varphi(u,\vec v)$, if $\mu(\l\varphi(\bo b,\vec{\bo a})\rr)>0$ then
$$\l\varphi(\bo b,\vec{\bo a})\rr=\l(\exists u)\varphi(u,\vec{\bo a})\rr.$$
\end{enumerate}
\end{cor}

\begin{proof}
$(\Rightarrow)$: Suppose $\bo b\in \dcl({ A})$.  (1) holds by Lemma \ref{l-pointwisedefinable}.
Suppose $\varphi(u,\vec v)$ is isolating and $\mu(\l\varphi(\bo b,\vec{\bo a})\rr)>0$.  We have
$\l\varphi(\bo b,\vec{\bo a})\rr\in\fo_\BB(\{\bo b\}\cup{ A})$, so
by Corollary \ref{c-dcl3}, $\l\varphi(\bo b,\vec{\bo a})\rr\in\dcl_\BB(A)$. By Theorem \ref{t-categorical},
$\l\varphi(\bo b,\vec{\bo a})\rr\in\fo_\BB(A)$.
We note that $(\exists u)\varphi(u,\vec v)$ is an isolating formula, so $\l(\exists u)\varphi(u,\vec{\bo a})\rr$ is an atom of $\fo_\BB(A)$.
Therefore (2) holds.

$(\Leftarrow)$:  Assume (1) and (2).  By (2), for every isolating formula $\varphi(u,\vec v)$ such that $\mu(\l\varphi(\bo b,\vec{\bo a})\rr)>0$,
we have
$$\l\varphi(\bo b,\vec{\bo a})\rr\in\fo_\BB(A).$$
Every formula $\theta(u,\vec v)$ is $T$-equivalent to a finite disjunction
of isolating formulas in the variables $(u,\vec v)$.  It follows that $\fo_\BB(A\cup\{\bo b\})\subseteq\fo_\BB(A)$.
Therefore by Corollary \ref{c-dcl3}, $\bo b\in\dcl(A)$.
\end{proof}

\begin{cor}  \label{c-cat-definable2}
Suppose that $T$ is $\aleph_0$-categorical, $\bo b\in\hat{\cu K}$, and $A$ is a finite subset of $\hat{\cu K}$.   Then
$\bo{b}\in \dcl({ A})$ if and only if for every isolating formula $\psi(\vec v)$ there is a functional
formula $\varphi(u,\vec v)$ such that $\l\psi(\vec{\bo a})\rr\sqsubseteq\l\varphi(\bo b,\vec {\bo a})\rr.$
\end{cor}

\begin{proof}
$(\Rightarrow)$: Suppose $\bo b\in \dcl({ A})$. By Theorem \ref{t-categorical}, $\bo b$ is first order definable over $\vec{\bo a}$,
so there is a functional formula $\varphi(u,\vec v)$ such that $\l\varphi(\bo b,\vec{\bo a})\rr=\top$.  Then for every isolating $\psi(\vec v)$
we have
$\l\psi(\vec{\bo a})\rr\sqsubseteq\l\varphi(\bo b,\vec{\bo a})\rr.$

$(\Leftarrow)$:  There is a finite set $\{\psi_0(\vec v),\ldots,\psi_k(\vec v)\}$ that contains exactly one isolating formula
for each $|\vec{\bo a}|$-type of $T$.  By hypothesis, for each $i\le k$ there is a functional formula $\varphi_i(u,\vec v)$
such that $\l\psi_i(\vec{\bo a})\rr\sqsubseteq\l\varphi_i(\bo b,\vec{\bo a})\rr.$
Since the formulas $\psi_i(\vec v)$ are pairwise inconsistent, the formula $\bigvee_{i=0}^k (\psi_i(\vec v)\wedge\varphi_i(u,\vec v))$ is functional, and
$$\l \bigvee_{i=0}^k (\psi_i(\vec{\bo a})\wedge\varphi_i(\bo b,\vec{\bo a}))\rr=\top.$$
Hence $\bo b$ is first order definable over $\vec{\bo a}$, so by Lemma \ref{l-firstorderdefinable} we have $\bo b\in \dcl({ A})$.
\end{proof}

\subsection{The Theory $\DLO^R$}

We will use Corollary \ref{c-cat-definable2} to give a more natural characterization of the definable closure of a finite parameter set
in a model of $\DLO^R$, where $\DLO$ is the theory of dense linear order without endpoints.  Note that in $\DLO$, every type in
$(v_1,\ldots,v_n)$ has an isolating
formula of the form $\bigwedge_{i=1}^{n-1} u_i\alpha_i u_{i+1}$ where $\{u_1,\ldots u_n\}=\{v_1,\ldots,v_n\}$ and each $\alpha_i\in\{<,=\}$.
(This formula linearly orders the equality-equivalence classes).

\begin{cor}  \label{c-DLO-definable}  Let $T=\DLO$, $\bo b\in\hat{\cu K}$, and $A$ be a finite subset of $\hat{\cu K}$. Then
$\bo{b}\in \dcl({ A})$ if and only if for every isolating formula $\psi(v_1,\ldots,v_n)$ there is an $i\in\{1,\ldots,n\}$
such that $\l\psi(\vec {\bo a})\rr\sqsubseteq\l \bo b= \bo a_i\rr.$
\end{cor}

\begin{proof}  For any  $\cu M\models\DLO$ and parameter set $A$, we have $\dcl^{\cu M}(A)=A$.  Therefore for every isolating formula
$\psi(v_1,\ldots,v_n)$ and functional formula $\varphi(u,v_1,\ldots,v_n)$ there exists $i\in\{1,\ldots,n\}$ such that
$$\DLO\models(\psi(v_1,\ldots,v_n)\wedge\varphi(u,v_1,\ldots,v_n))\rightarrow u=v_i.$$
The result now follows from Corollary \ref{c-cat-definable2}.
\end{proof}

In the theory $\DLO$, we define $\min(u,v)$ and $\max(u,v)$ in the usual way.  For $\bo a,\bo b\in\hat{\cu K}$, we let $\min(\bo a,\bo b)$ be the unique
element $\bo e\in\hat{\cu K}$ such that
$$\l e=\min(a,b)\rr=\top,$$
and similarly for $\max$.  For finite subsets $A$ of $\hat{\cu K}$, $\min(A)$ and $\max(A)$ are defined by repeating the two-variable functions
$\min$ and $\max$ in the natural way.

We next show that in $\DLO^R$, the definable closure of a finite set can be characterized as the closure under a ``choosing function'' of four variables.

\begin{df}  In the theory $\DLO$, let $\ell$ be the function of four variables defined by the condition
$$ \ell(u,v,x,y)=x \mbox{ if } u < v, \mbox{ and } \ell(u,v,x,y) = y \mbox{ if not } u < v.$$
For $\bo a,\bo b,\bo c,\bo d\in\cu K$, let  $\ell(\bo a,\bo b,\bo c,\bo d)$ be the unique element $\bo e\in\hat{\cu K}$ such that
$\l e=\ell(a,b,c,d)\rr=\top$.
Given a set $ A\subseteq\hat{\cu K}$, let $\lcl( A)$ be the closure of $ A$ under the function $\ell$.
\end{df}

Note that in $\DLO$, the function $\ell$ is definable without parameters.  In both $\DLO$ and $\DLO^R$, $\min(u,v)=\ell(u,v,u,v)$, and $\max(u,v)=\ell(u,v,v,u)$.

\begin{prop}  \label{p-DLO}
Let $T=\DLO$.  Then for every finite subset $A$ of $\hat{\cu K}$, $\dcl( A)=\lcl( A)$.
\end{prop}

\begin{proof}  It is clear that $\lcl( A)\subseteq\dcl( A)$.

We prove the other inclusion.  If $A$ is empty, the result is trivial, so we assume $A$ is non-empty.
Let $\bo 0=\min(A), \bo 1=\max(A)$.  We have $\bo 0, \bo 1\in\lcl( A)$.
Let $\Omega_0=\l 0<1\rr$.  Note that $\Omega\setminus\Omega_0=\l 0=1\rr$.  If $\mu(\Omega_0)=0$, then $ A$ is a singleton,
and we trivially have $\lcl( A)=\dcl( A)= A$.  We may therefore assume that $\mu(\Omega_0)>0$.  To simplify notation we will
instead assume that $\Omega_0=\Omega$; the argument in the general case is similar.

In the following, all characteristic functions are understood to be with respect to $\bo 0, \bo 1$.
Note that $\ell(\bo a,\bo b,\bo 0,\bo 1)$ is the characteristic function of the event $\l \bo a <\bo b\rr$.
If $\bo d$ is the characteristic function of an event $\sa D$ and $\bo e$ is the characteristic function of an event $\sa E$,
then $\ell(\bo d,\bo 1,\bo 1,\bo 0)$ is the characteristic function of $\neg\sa D$, $\min(\bo d,\bo e)$ is the characteristic function
of $\sa D\sqcap\sa E$, and $\max(\bo d,\bo e)$ is the characteristic function of $\sa D\sqcup\sa E$.  It follows that for every
quantifier-free first order formula $\varphi(\vec v)$ of $\DLO$ with $|\vec v|=|\vec{\bo a}|$, the characteristic
function of the event $\l\varphi(\vec{\bo a})\rr$ belongs to $\lcl(A)$.  Since $\DLO$ admits quantifier
elimination, the characteristic function of every event that is first order definable over $A$ belongs to $\lcl( A)$.
Hence by Theorem \ref{t-categorical}, the characteristic function of every event in $\dcl_\BB(A)$ belongs to $\lcl(A)$.
Moreover, for every $\bo c\in A$ and event $\sa D\in\dcl_\BB(A)$ with characteristic function $\bo d$,
$\bo c\upharpoonright{\sa D}:=\ell(\bo d,\bo 1,\bo 0,\bo c)$ is the element that agrees with $\bo c$ on
$\sa D$ and agrees with $\bo 0$ on the complement of $\sa D$, so $\bo c\upharpoonright{\sa D}$ belongs to $\lcl( A)$.
Let $\{\sa D_1,\ldots,\sa D_n\}$ be the set of atoms of $\dcl_\BB( A)$ (which is finite because $\DLO$ is $\aleph_0$-categorical).
By Corollary \ref{c-DLO-definable}, every element of $\dcl( A)$ has the form
$$ \max(\bo c_1\upharpoonright\sa D_1,\ldots,\bo c_n\upharpoonright\sa D_n)$$
for some $\bo c_1,\ldots,\bo c_n\in A$.  Therefore $\dcl( A)\subseteq\lcl( A)$.
\end{proof}

\begin{ex}  In this example we show that the exchange property fails for $\DLO^R$, even though it holds for $\DLO$.  Thus
the exchange property is not preserved under randomizations.  Let $T=\DLO$.  By Fullness, there exist elements $\bo a, \bo b\in\hat{\cu K}$
such that $\max(\bo a,\bo b)\notin\{\bo a,\bo b\}$.  Let $\bo c=\max(\bo a,\bo b), \bo d=\min(\bo a,\bo b)$.
It is easy to check that
$$\dcl(\{\bo a,\bo b\})=\{\bo a,\bo b,\bo c,\bo d\}, \quad \dcl(\{\bo a,\bo c\})=\{\bo a,\bo c\},\quad \dcl(\{\bo a\})=\{\bo a\}.$$
Thus $\bo c\in\dcl(\{\bo a,\bo b\})\setminus\dcl(\{\bo a\})$ but $\bo b\notin\dcl(\{\bo a,\bo c\})$.
\end{ex}

\section*{References}


\vspace{2mm}

[AK]  Uri Andrews and H. Jerome Keisler.  Randomizations of Theories with Countably Many
Countable Models.  To appear.  Available online at www.math.wisc.edu/$\sim$Keisler.

[Be]  Ita\"i{} Ben Yaacov.  On Theories of Random Variables.  To appear, Israel J. Math.  ArXiv:0901.1584v3 (2001).

[BBHU]  Ita\"i{} Ben Yaacov, Alexander Berenstein,
C. Ward Henson and Alexander Usvyatsov. Model Theory for Metric Structures.
To appear, Lecture Notes of the London Math. Society.

[BK] Ita\"i{} Ben Yaacov and H. Jerome Keisler. Randomizations of Models as Metric Structures.
Confluentes Mathematici 1 (2009), pp. 197-223.

[BU] Ita\"i{} Ben Yaacov and Alexander Usvyatsov. Continuous first order logic and local stability. Transactions of the American
Mathematical Society 362 (2010), no. 10, 5213-5259.

[CK]  C.C.Chang and H. Jerome Keisler.  Model Theory.  Dover 2012.

[EG]  Clifton Early and Isaac Goldbring.  Thorn-Forking in Continuous Logic.  Journal of Symbolic Logic
77 (2012), 63-93.

[Go1]  Isaac Goldbring.  Definable Functions in Urysohn's Metric Space.  To appear, Illinois Journal of Mathematics.

[Go2]  Isaac Goldbring.  An Approximate Herbrand's Theorem and Definable Functions in Metric Structures.
Math. Logic Quarterly 50 (2012), 208-216.

[Go3] Isaac Goldbring.  Definable Operators on Hilbert Spaces.  Notre Dame Journal of Formal Logic 53 (2012), 193-201.

[GL]  Isaac Goldbring and Vinicius Lopes.  Pseudofinite and Pseudocompact Metric Structures.  To appear, Notre Dame Journal of Formal Logic.
Available online at www.homepages.math.uic.edu/$\sim$isaac.

[Ke1] H. Jerome Keisler.  Randomizing a Model.  Advances in Math 143 (1999),
124-158.

[Ke2] H. Jerome Keisler.  Separable Randomizations of Models.  To appear.  Available online at
www.math.wisc.edu/$\sim$Keisler.

\end{document}